\documentclass[12pt]{amsart}

\usepackage{amssymb,amsthm,amsmath}
\usepackage[numbers,sort&compress]{natbib}
\usepackage{color}
\usepackage{graphicx}
\usepackage{amsthm}
\usepackage{physics}
\usepackage{enumitem}
\usepackage{tikz}
\usepackage{amssymb,amsthm,amsmath, subcaption}
\usepackage[numbers,sort&compress]{natbib}
\usepackage{color}
\DeclareMathOperator{\sgn}{sgn}
\usepackage{graphicx}
\usepackage{tikz}
\usetikzlibrary{shapes}
\usepackage{xparse}
\usetikzlibrary{calc,arrows.meta,positioning}

\title[Floquet isospectrality of The Zero Potential]{ Floquet isospectrality of The Zero Potential for discrete periodic Schr\"odinger operators}
\author[M. Faust]{Matthew Faust}
\address[M. Faust]{ Department of Mathematics, Texas A\&M University, College Station, TX 77843-3368, USA} \email{mfaust@tamu.edu}
\author[W. Liu]{Wencai Liu}
\address[W. Liu]{ Department of Mathematics, Texas A\&M University, College Station, TX 77843-3368, USA} \email{liuwencai1226@gmail.com; wencail@tamu.edu}
\author[R. Matos]{Rodrigo Matos}
\address[R. Matos]{ Department of Mathematics, PUC-Rio\newline  Rio de Janeiro, Brazil, 22451-900} \email{rodrigo@mat.puc-rio.br}
\author[J. Plute]{Jenna Plute}
\address[J. Plute]{ Department of Mathematics, Texas A\&M University, College Station, TX 77843-3368, USA} \email{jplute@tamu.edu}
\author[J. Robinson]{Jonah Robinson}
\address[J. Robinson]{ Department of Mathematics, Texas A\&M University, College Station, TX 77843-3368, USA} \email{jonahrobinson@tamu.edu}
\author[Y. Tao]{Yichen Tao}
\address[Y. Tao]{ Department of Mathematics, Texas A\&M University, College Station, TX 77843-3368, USA} \email{yichentao@tamu.edu}
\author[E. Tran]{Ethan Tran}
\address[E. Tran]{ Department of Mathematics, Texas A\&M University, College Station, TX 77843-3368, USA} \email{et.tran50@tamu.edu}
\author[C. Zhuang]{Cindy Zhuang}
\address[C. Zhuang]{ Department of Mathematics, Texas A\&M University, College Station, TX 77843-3368, USA} \email{cindylz@tamu.edu}

\keywords{Ambarzumyan-type  problem, Floquet isospectrality, discrete periodic Schr\"odinger operators.}

\thanks{{\em 2020 Mathematics Subject Classification.} Primary: 58J53. Secondary: 47B36, 35P05, 35J10.}

\theoremstyle{plain}
\newtheorem{theorem}{Theorem}[section]
\newcommand{\R}{\mathbb{R}}

\newtheorem{lemma}[theorem]{Lemma}

\newtheorem{remark}{Remark}
\newcommand{\C}{\mathbb{C}}

\newcommand{\Z}{\mathbb{Z}}

\theoremstyle{plain}
\newtheorem{definition}{Definition}
\newtheorem{conjecture}{Conjecture}

\newcommand{\defcolor}[1]{{\color{blue}#1}}
\newcommand{\demph}[1]{\defcolor{{\sl #1}}}

\def\rod#1{#1}

\begin{document}
	
	
	\begin{abstract}
	 
	Let  $\Gamma=q_1\mathbb{Z}\oplus q_2 \mathbb{Z}\oplus\cdots\oplus q_d\mathbb{Z}$, with $q_j\in (\mathbb{Z}^+)^d$ for each $j\in \{1,\ldots,d\}$, and denote by $\Delta$ the discrete Laplacian on $\ell^2\left( \mathbb{Z}^d\right)$.
Using Macaulay2, we first numerically find  complex-valued $\Gamma$-periodic potentials $V:\mathbb{Z}^d\to \mathbb{C}$  such that the operators
  $\Delta+V$ and $\Delta$ are Floquet isospectral. We then use combinatorial methods to validate these numerical solutions.

	\end{abstract}
	
	\maketitle 
	\section{Introduction and main results}
 Let $\Delta$ denote the discrete Laplacian on $\Z^d$:
\begin{equation*}
	(\Delta u)(n)=\sum_{|n^\prime-n|_1=1}u(n^\prime),
	\end{equation*}
where for $n=(n_1,n_2,\cdots,n_d)$, $n^\prime=(n_1^\prime,n_2^\prime,\cdots,n_d^\prime)\in\Z^d$ we write \begin{equation*}
	|n^\prime-n|_1:=\sum_{i=1}^d |n_i-n^\prime_i|.
	\end{equation*}

In this paper we study finite difference equations of the type
\begin{equation}~\label{eq1}
(\Delta u)(n) + V(n)u(n) = \lambda u(n),\,\,\text{for all}\, n \in \Z^d,
\end{equation}
subject to the Floquet boundary condition
\begin{equation}~\label{eq2}
u(n + q_je_j) = e^{2 \pi i k_j} u(n), \text{ for all } n \in \Z^d \,\,,k_j\in [0,1]\,\, \text{and } j \in \demph{[d] := \{1,\dots, d\}}
\end{equation}
where 
$\{e_j\}_{j=1}^d$ is the standard basis in $\R^d$ \rod{and $\{q_j\}^d_{j=1}$ are positive integers}. The potential $V:\Z^d \to \C$ is assumed to be $\Gamma$-periodic, namely: letting $\Gamma = q_1\Z \oplus \dots \oplus q_d \Z$ we impose that
\begin{equation*}V(n+m)=V(n),\,\,\text{for all}\,\, n\in \mathbb{Z}^d\,\,\text{and}\,\,m\in \Gamma.
\end{equation*}

Writing $Q = \prod_{i=1}^d q_i$, equation~\eqref{eq1} with the boundary condition~\eqref{eq2} can be realized as the eigen-equation of a finite $Q \times Q$ matrix $D_V(k)$, where $k = (k_1, \dots, k_d)$. Denote by $\sigma_V(k)=\sigma(D_V(k))$ the set of eigenvalues of $D_V(k)$, including algebraic multiplicity. 
\begin{definition}
Assume that $V$ and $V'$ are $\Gamma$-periodic potentials. Two operators $\Delta + V$ and $\Delta + V'$ are called \demph{Floquet isospectral} if $\sigma_V(k) = \sigma_{V'}(k)$ for all $k \in \R^d$. In this case it is also common to say that the $\Gamma$-periodic potentials $V$ and $V'$ are Floquet isospectral.
\end{definition}

 Understanding when two potentials $V$ and $V'$ are Floquet isospectral is an interesting problem with a rich history of study ~\cite{kapiii,Kapi,Kapii,liujde}. In this paper we focus on the discrete case, however there is also a deep body of work for the continuous case ~\cite{ERT84,MT76,ERTII,gki,wa,gui90,eskin89}.
 
   Floquet isospectrality belongs to the class of  inverse spectral  problems, \rod{where one aims at recovering information about the original operator from certain spectral data. Among such problems we also highlight} Fermi isospectrality~\cite{liu2021fermi,bktcm91}, Borg's Theorem\cite{liuborg,Ges1, borg,s2011} and \rod{recovery of the potential from its integrated density of states} \cite{GKTBook}. For more background and history of inverse spectral problems for periodic potentials we refer the reader to the following recent surveys~\cite{kuchment2023analytic,liujmp22,ksurvey}.

 In this paper we focus on the Ambarzumyan-type inverse problem of finding potentials isospectral to the zero potential, denoted henceforth by $\bf{0}$. It is a well-known and classical result that there no real non-zero potentials Floquet isospectral to $\bf{0}$ in both the continuous and discrete cases. For more general work on the well-studied class of Ambarzumyan-type problems we refer the reader to \cite[Chapter 14]{kur1}. For the continuous case it has been shown that there are many complex-valued potentials that are  Floquet isospectral to $\bf{0}$ (e.g. \cite{GU}). It is a Folklore result that there exist
 complex-valued $\Gamma$-periodic potentials $V$ which are Floquet isospectral to $\bf{0}$ in the discrete case. Our goal of this paper is multi-fold. Firstly we present explicit potentials isospectral to $\bf{0}$. \rod{Secondly, our method of proof is, to the best of our knowledge, new. More concretely, we introduce combinatorial language and techniques which might be of interest to other problems in the realm of the spectral theory of discrete  Schr\"odinger operators.}

\medskip
\begin{theorem}~\label{THM:1}
    Let  $\Gamma=q_1\mathbb{Z}\oplus q_2 \mathbb{Z}\oplus\cdots\oplus q_d\mathbb{Z}$. Assume that at least one of $q_j$, $j=1,2,\cdots$ is even.
    Then there exist a nonzero $\Gamma$-periodic function $V$ Floquet isospectral to $\bf{0}$.
\end{theorem}

\begin{remark}
We prove Theorem~\ref{THM:1} by constructing explicit potentials isospectral to $\bf{0}$ (See Theorem~\ref{THM:3.2}).
\end{remark}

In order to construct these explicit examples, we began by experimentally finding solutions using Macaulay2. Although we do not use all the numerical solutions found, our experimental data enabled us to notice a particular pattern for when potentials are Floquet isospectral to $\bf{0}$.  In this paper, we begin with this pattern and focus on proving it. \rod{The reader interested in how this pattern was discovered by us is invited to consult our annotated Macaulay2 code at the github} repository\footnote{https://mattfaust.github.io/IsoZ/IsoZ.m2}. By Floquet theorem, the discrete Sch\"odinger operator can be represented as the direct sum of a family of  finite square matrices. By modeling this matrix as a finite directed graph, we are able to use graph theory and algebra to verify the solutions suggested by the observed numerical pattern.

\rod{The remainder of this note is organized as follows}. In Section~\ref{SEC:2} we give a brief background on the combinatorial constructions we will employ. In Section~\ref{SEC:3} we reduce Theorem~\ref{THM:1} to Theorem~\ref{THM:3.2}. In Section~\ref{SEC:4} we prove Theorem~\ref{THM:3.2}.

    \section{Combinatorial Background and Notation}~\label{SEC:2}
    In this section we introduce basic terminology and also make a few remarks which will be relevant to the proof of the main results of this note.  Let us start by introducing some notation.
    \begin{enumerate}[label=(\roman*), series=list1]
  \item Fix a $m \times m$ matrix $M=\left(M_{ij}\right)_{m\times m}$ with entries in a ring $R$. The digraph of $M$ is a weighted directed graph $G:=(\mathcal{V},\mathcal{E},w)$ with vertex set $\mathcal{V}=[m]$ and edges given by $(i,j) \in \mathcal{E}$ if $M_{i,j} \neq 0$. \rod{The weight function $w : \mathcal{E} \to R$ is naturally inherited from $M$ via $w(i,j) = M_{i,j}$ for all $(i,j)\in \mathcal{E}$.}
 \rod{ \item A vertex cycle cover is a union of cycles which are subgraphs of $G$ and contains all vertices of $G$.} 
  \item  \rod{A \demph{disjoint cycle cover} of a digraph is a vertex cycle cover for which two different cycles have no vertices in common}.
  
  \item We denote by $\mathcal{S}_m$ the symmetric group of $m$ elements. The permutations in $\mathcal{S}_m$ will be denoted by $\sigma$.
  \item We also let $M_\sigma= \prod_{i=1}^m M_{i,\sigma(i)}$.
  \item Since any permutation $\sigma \in S_m$ is a product of disjoint cycles, it is easy to see that if $M_\sigma\neq 0$ then $\sigma$ induces a disjoint cycle cover $\eta$ of $G$ such that $$w(\eta) := \prod_{e \in \eta} w(e) = M_{\sigma}$$ We then denote $\sgn(\eta):= \sgn(\sigma)$.
  \item We let 
\begin{equation*}\mathcal{P}=\{\sigma\in \mathcal{S}_m :\,\,M_\sigma \neq 0\}. \end{equation*}
Note that $\mathcal{P}$ is in bijection with the collection of disjoint cycle covers of $G$, denoted henceforth by $\mathcal{U}_G$.
\end{enumerate}
We now make the following remarks:

\begin{enumerate}[label=(R\arabic*)]
\item Notice that the digraph of $M$ is a directed version of the adjacency matrix. 
\vspace{0.1cm}
\item With the given definitions, one has that
\begin{equation}
    \label{eqn:det trans}
        \det(M) = \sum_{\sigma \in P} \sgn(\sigma) M_\sigma  = \sum_{\eta \in U} \sgn(\eta) w(\eta).
    \end{equation}
\end{enumerate}

\begin{enumerate}[resume=list1, label=(\roman*)]
\item    Let $\demph{[t]}f$ denote the coefficient of the monomial term $t$ of a polynomial $f$. For example if $f(z) = 4 z^2 + 9 z - 2$ then  $[z^2]f(z) = 4, [z] f(z) = 9,$ and $[z^0] f(z) = -2$.

\item We say that $M=\left(M_{ij}\right)_{m\times m}$ is a \demph{Jacobi matrix} if it has the following property: \rod{If $M_{ij}\neq 0$ for some $i,j\in [m]$ then $\abs{i-j}\leq 1$. In this case we denote by} \demph{$J_m$} be the digraph of $M$, as shown in Figure~\ref{Fig:1}. It is easy to see that if $\eta$ is a disjoint cycle cover of $J_m$, $\eta$ must be composed of only cycles of length $1$ and $2$.

\item  Let \demph{$S(m,p)$} denote the number of disjoint cycle covers of $J_m$ with exactly $p$ $2$-cycles. This quantity will be used in the proof of Theorem~\ref{THM:3}. Although we do not need the explicit value of $S(m,p)$, by basic combinatorial identities one finds that $S(m,p) = \binom{m-p}{p}$.
\end{enumerate}

\begin{figure}[ht]
 \begin{tikzpicture}[->]
    \node[ellipse,draw] (a) at (0,0) {1};
    \node[ellipse,draw] (b) at (4,0) {2};
    \node[ellipse,draw] (d) at (8,0) {m-1};
    \node[ellipse,draw] (e) at (12,0) {m};      
    
    \path (a) edge[bend right]              node[below] {$M_{1,2}$} (b);
    \path (b) edge[bend right]           node[above] {$M_{2,1}$} (a);
    \node (i) at ($(b)!.5!(d)$) {\dots};
    \path (b) edge[bend right]              node[below] {$M_{2,3}$} (i);
    \path (i) edge[bend right]           node[above] {$M_{3,2}$} (b);
     \path (i) edge[bend right]              node[below] {$M_{m-2,m-1}$} (d);
    \path (d) edge[bend right]           node[above] {$M_{m-1,m-2}$} (i);
    \path (d) edge[bend right]              node[below] {$M_{m-1,m}$} (e);
    \path (e) edge[bend right]           node[above] {$M_{m,m-1}$} (d);
    \path (a) edge [loop below] node {$M_{1,1}$} (a);
    \path (b) edge [loop below] node {$M_{2,2}$} (b);
    \path (d) edge [loop below] node {$M_{m-1,m-1}$} (d);
    \path (e) edge [loop below] node {$M_{m,m}$} (e);
  \end{tikzpicture} 
\caption{The weighted digraph $J_m$ of an $m \times m$ Jacobi matrix $M$.}
\label{Fig:1}
\end{figure}
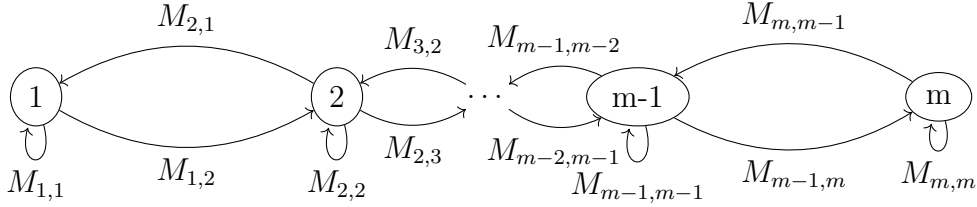

\section{Reduction of Main Theorem}~\label{SEC:3}

To prove Theorem~\ref{THM:1}, we provide an explicit potential \rod{which is} Floquet isospectral to $\bf{0}$. 
 
 Suppose that $m \in \Z^+$. Let us define \rod{a vector $v=(v_1,\ldots,v_{2m})\in \C^{2m}$ via} \begin{equation}\label{eq:potential}
 v \in \C^{2m}, v_1 = 1+i, v_2 = 1-i, v_{m+1} = -1+i, v_{m+2} = -1-i \text{ and } v_l = 0 \text{ otherwise.} 
 \end{equation}  

 As a $2m\Z$-periodic potential $r$ is determined by the values $r_1 := r(1)$ $,\dots,$ $r_{2m} := r(2m)$, we often abuse notation and identify $r$ with the vector $r=(r_1,\dots, r_{2m})$.

\begin{theorem}~\label{THM:3} \rod{Let  $m \in \Z^+$} and $\Gamma = 2m \Z \oplus q_2 \Z \oplus \dots \oplus q_d \Z$, and define $V : \Z^d \to \C$ be the separable $\Gamma$-periodic potential given by $V(n_1,\dots, n_d) = v_{(n_1 \text{ mod }{2m})}$ \rod{with $v$ given by \eqref{eq:potential}. Then} $V$ is Floquet isospectral to $\bf{0}$.
\end{theorem}
\rod{Theorem~\ref{THM:3} clearly implies Theorem~\ref{THM:1}}

To prove Theorem~\ref{THM:3} we need only show the following theorem.

\begin{theorem}~\label{THM:3.2} Let $v$ be given by ~\eqref{eq:potential}, then the following two matrices have the same eigenvalues, \rod{including multiplicity} 
        \begin{equation}\label{eq:defmatrices}
        D_{v}=\begin{pmatrix}
            v_1 & 1 & 0 & \dots & 0 & 1 \\
            1 & v_2 & 1 & 0 & \dots & 0 \\
            0 & 1 & \ddots & \ddots & \ddots & \vdots \\
            \vdots & \ddots & \ddots & \ddots & \ddots & 0\\ 
            0 & 0 & \dots & 1 & v_{2m-1} & 1 \\
            1 & 0 & \dots & 0 & 1 & v_{2m} \\
        \end{pmatrix}\\
      ,\,\,\, D_{\bf{0}}=\begin{pmatrix}
            0\, & 1 & 0 & \dots & 0 & 1 \\
            1 & 0\, & 1 & 0 & \dots & 0 \\
            0 & 1 & \ddots\, & \ddots & \ddots & \vdots \\
            \vdots & \ddots & \ddots & \ddots\, & \ddots & 0\\ 
            0 & 0 & \dots & 1 & 0\, & 1 \\
            1 & 0 & \dots & 0 & 1 & 0\, \\
        \end{pmatrix}.\\
        \end{equation}
\end{theorem} 

\begin{definition}
    A function $V : \Z^d \to \C$ is called completely separable if there exists $V_j: \Z \to \C$ such that $V(n) = V(n_1,\dots, n_d) = V_1(n_1) + \dots + V_d(n_d)$ for all $n \in \Z^d$. If $V$ is completely separable we write $V = V_1 \oplus \dots \oplus V_d$.
\end{definition}

\begin{lemma}~\label{LEM:Sep}
    Let $\Gamma = q_1 \Z \oplus \dots \oplus q_d\Z$. Assume $V = V_1 \oplus \dots \oplus V_d$ and $W = W_1 \oplus \dots \oplus W_d$ are completely separable $\Gamma$-periodic potentials. If $W_j$ and $V_j$ are Floquet isospectral for each $j\in [d]$ then $V$ and $W$ are Floquet isospectral. 
\end{lemma}
\begin{proof}
\rod{This follows from a well-known} result e.g. \cite[Theorem 4.14]{GTES}.
\end{proof}

\begin{proof}[Proof of Reduction of Theorem 3.1 to Theorem 3.2]
Notice that $V$ and $\bf{0}$ are both completely separable potentials. By Lemma ~\ref{LEM:Sep}, to prove Theorem~\ref{THM:3} it suffices to show that the $2m \Z$-periodic potential $v$ is isospectral to the $2m \Z$-periodic $\bf{0}$. Furthermore, for the one dimensional Schr\"odinger operator, it is well-known that the Floquet isospectrality of two potentials is determined by their values at a single quasi-momenta(e.g. ~\cite[page 1]{Kapii}). Let us fix the quasi-momtenta $k = 0$. Then $V$ and $\bf{0}$ are isospectral if $D_v$ and $D_{\bf{0}}$ have the same eigenvalues, each with the same multiplicity for each respective matrix.
\end{proof}

\section{Combinatorial Formulation and Proof of Theorem~\ref{THM:3.2}}~\label{SEC:4}

\rod{Let $P_v(\lambda)$ and $P_{\bf{0}}(\lambda)$ denote the characteristic polynomials of $D_v$ and $D_{\bf{0}}$, respectively. In order to prove Theorem \ref{THM:3.2} it suffices to check that
\begin{equation}\label{eq:characpol}
P_v(\lambda)= P_{\bf{0}}(\lambda)\,\,\text{for all}\,\,\lambda\in \mathbb{C}.
\end{equation}
Before taking $v$ to have the specific form \eqref{eq:potential} we investigate the difference $P_v(\lambda)-P_{\bf{0}}(\lambda)$ when \[v = (v_1, v_2, 0, \dots, 0, v_{3}, v_{4}, 0, \dots, 0).\] for arbitrary values of $v_1,\ldots,v_4$. Note that in this case the polynomial $P_v(\lambda)-P_{\bf{0}}(\lambda)$ has degree at most $\lambda^{2m-1}$ and each of its coefficients must depend on at least one of the variables $v_1,\ldots,v_4$. More can be said and, in fact, we have the following complete description of the above difference.  Define \[ F_i(v) := [\lambda^i] \left(P_v(\lambda) - P_{\bf{0}}(\lambda)\right) = 0 \text{ for }i \in \{ 0,\dots, 2m-1\}.\]
}

\begin{theorem}\label{thm:Fnformula}
The following formulas for $F_{2m-k}(v)$ hold:

\[F_{2m-k}(v) = v_1v_2v_3v_4  (-1)^{\ell} (\sum_{i = 0}^{\ell-2} S(m-2, i) S(m-2, \ell-2-i)) + (v_1 v_2 + v_3 v_4) (-1)^{\ell-1} S(2m-2,\ell-1)\] \[ + (v_1 v_3 + v_2 v_4) (-1)^{\ell-1} (\sum_{i = 0}^{\ell-1} S(m-1, i) S(m-1 , \ell-1-i)) \] \[ + (v_1 v_4 + v_2 v_3) (-1)^{\ell-1} (\sum_{i = 0}^{\ell-1} S(m, i) S(m-2, \ell-1-i)), \]
if $k=2\ell$ for $\ell \in [m-1]$.
\[ F_{2m-k}(v) = (v_1v_2v_3+v_1v_2v_4+v_1v_3v_4+v_2v_3v_4)  (-1)^{\ell} (\sum_{i = 0}^{\ell-1} S(m-2, i) S(m-1, \ell-1-i)) \] \[ + (v_1 +v_2 + v_3 +v_4) (-1)^{\ell+1} S(2m-1,\ell),\] 
if $k=2\ell+1$ for $\ell \in [m]$.
    
\end{theorem}

\begin{remark}~\label{RM:2} \rod{It readily follows that, for the above range of values for $\ell$, the parity of $k$ and any of the monomials appearing in $[F_{2m-k}(v)$ is the same.} 
\end{remark}
\begin{proof}

Here we spell out the details for the case where $k$ is even.The proof for odd values of $k$ is completely analogous.
Thus, let $k = 2\ell$ for some $\ell \in \Z^+$ \rod{$l\in [m-1]$}. 
Let $G$ be the digraph of $D_v(z)-\lambda I$ and let $\mathcal{U}_{G}$ the collection of disjoint cycle covers of $G$. By \eqref{eqn:det trans} we have that $$P_v(\lambda) = \sum_{\eta \in \mathcal{U}_G} \sgn(\eta) w(\eta).$$ Define $G' = (\mathcal{V},\mathcal{E}',w')$ to be a refined version of $G$, where edges \rod{of the form} $(i,i)$ are replaced by \rod{two new edges} $(i,i)_1$ and $(i,i)_2$ with weights $w'((i,i)_1) = v_i$, $w'((i,i)_2) = -\lambda$.\rod{ All other edges and weight assignments are kept fixed} (see Figure~\ref{Fig:2}). Let $\mathcal{U}_{G'}$ be the collection of disjoint cycle covers of $G'$. Notice that if $\eta \in \mathcal{U}_{G}$ does not contain $(i,i)$ then $\eta \in \mathcal{U}_{G'}$ and $w(\eta) = w'(\eta)$. Further given an $\eta \in \mathcal{U}_{G}$ with a single self loop edge $(i,i)$ there exists disjoint cycle covers $\alpha$ and $\beta$ in $\mathcal{U}_{G'}$ such that $w(\eta) = w'(\alpha) + w'(\beta)$. In general $\eta \in \mathcal{U}_{G}$ splits into $2^l$ disjoint cycle covers $\alpha_1,\dots,\alpha_{2^l}$ of $G'$ where $l$ is the number of self loops in $\eta$, such that $w(\eta) = \sum_{i=1}^{2^l} w(\alpha_i)$. Thus we have that $$P_v(\lambda) = \sum_{\eta \in \mathcal{U}_{G'}} \sgn(\eta) w'(\eta).$$

\begin{figure}[ht]
 \begin{tikzpicture}[->]
    \node[ellipse,draw] (a) at (0,0) {1};
    \node[ellipse,draw] (b) at (2,0) {2};
    \node[ellipse,draw] (c) at (4,0) {3};
    \node[ellipse,draw] (d) at (5.4,0) {m};
    \node[ellipse,draw] (e) at (7.4,0) {m+1};   
    \node[ellipse,draw] (f) at (9.4,0) {m+2};   
    \node[ellipse,draw] (g) at (11.4,0) {m+3};   
    \node[ellipse,draw] (h) at (-2,0) {2m};   

    \node at (-2.9,0) {\dots};
    \path (a) edge[bend right]              node[below] {1} (b);
    \path (b) edge[bend right]           node[above] {1} (a);
    \path (b) edge[bend right]              node[below] {1} (c);
    \path (c) edge[bend right]           node[above] {1} (b);
    \node at (4.7,0) {\dots};
    \path (d) edge[bend right]              node[below] {1} (e);
    \path (e) edge[bend right]           node[above] {1} (d);
    \path (e) edge[bend right]              node[below] {1} (f);
    \path (f) edge[bend right]           node[above] {1} (e);
    \path (f) edge[bend right]              node[below] {1} (g);
    \path (g) edge[bend right]           node[above] {1} (f);
    \node at (12.5,0) {\dots};
    \path (a) edge [loop below] node {$v_1$} (a);
    \path (b) edge [loop below] node {$v_2$} (b);
    \path (a) edge [loop above] node {$-\lambda$} (a);
    \path (b) edge [loop above] node {$-\lambda$} (b);
    \path (c) edge [loop above] node {$-\lambda$} (c);
    \path (d) edge [loop above] node {$-\lambda$} (d);
    \path (e) edge [loop below] node {$v_3$} (e);
    \path (f) edge [loop below] node {$v_4$} (f);
    \path (e) edge [loop above] node {$-\lambda$} (e);
    \path (f) edge [loop above] node {$-\lambda$} (f);
    \path (g) edge [loop above] node {$-\lambda$} (g);
    \path (h) edge [loop above] node {$-\lambda$} (h);
    \path (a) edge[bend left]              node[below] {1} (h);
    \path (h) edge[bend left]           node[above] {1} (a);
  \end{tikzpicture} 
\caption{The refined digraph $G'$. Notice we do not include the weight $0$ self loops after specializing.}
\label{Fig:2}
\end{figure}
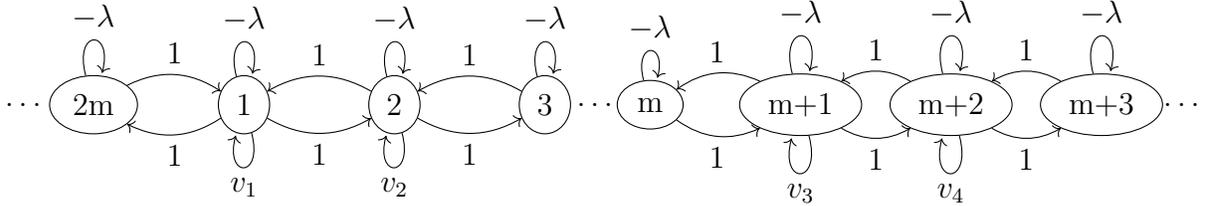

To solve for $[v_1v_2v_3v_4 \lambda^{2m-k}]P_v(\lambda)$, consider the collection \rod{{$\mathcal{C}$} of $\eta \in \mathcal{U}_{G'}$ with the following properties
\begin{itemize}
\item $w'(\eta) \neq 0$.
    \item $\eta$ contains the edges $(1,1)_1$, $(2,2)_1$, $(m+1,m+1)_1$, $(m+2,m+2)_1$.
    \item $\eta$ has $2m-k$ other edges of the form $(i,i)_2$.
\end{itemize}
}

Given an $\eta \in \mathcal{C}$, after fixing the edges $(1,1)_1$, $(2,2)_1$, $(m+1,m+1)_1$, and $(m+2,m+2)_1$, the remaining digraph $G' \smallsetminus \{1,2,m+1,m+2\}$ is a disjoint union of two $J_{(m-2)}$ digraphs, denoted $G'_1$ and $G'_2$. It follows that the remaining cycles of $\eta$ are the $2m-k$ $1$-cycles given by $(i,i)_2$ and $\ell-2$ $2$-cycles. Thus 
\rod{\begin{equation}\label{eq:coeffsingledeg4}
[v_1v_2v_3v_4 \lambda^{2m-k}] \sgn(\eta) \eta_{w'}=(-1)^{\ell}.    
\end{equation}
In particular,
\begin{equation}\label{eq:fullcoeffdeg4}
[v_1v_2v_3v_4 \lambda^{2m-k}]P_v(\lambda)= (-1)^{\ell}\lvert \mathcal{C} \rvert,
\end{equation}
and hence we are left to count the elements of $T$. Notice that after fixing the edges $(1,1)_1$, $(2,2)_1$, $(m+1,m+1)_1$, and $(m+2,m+2)_1$, the remaining $1$-cycles of $\eta$ in $\mathcal{C}$ are determined by the $2$-cycles of $\eta$. Thus to count $T$, we only need to count the number of way we can choose $\ell-2$ $2$-cycles from the disjoint subdigraphs $G'_1$ and $G'_2$. As we must choose $\ell-2$ $2$-cycles, if we choose $i$ $2$-cycles from $G'_1$ we must choose $\ell-2 - i$ $2$-cycles from $G'_2$. 
It follows that 
\begin{equation}\label{eq:Tcard}
\lvert \mathcal{C} \rvert = \sum_{i=1}^{\ell-2} S(m-2,i)S(m-2,\ell-2-i).
\end{equation}
Combining equations \eqref{eq:coeffsingledeg4}, \eqref{eq:fullcoeffdeg4} and \eqref{eq:Tcard} we conclude that
\[[v_1v_2v_3v_4]F_{2m-k} = (-1)^\ell \sum_{i=1}^{\ell-2} S(m-2,i)S(m-2,m-2-i).\]}

The other cases follow through similar arguments. See Figure~\ref{Fig:3} for an illustration of why the coefficients of $v_1v_2$ and $v_3v_4$ agree.
\end{proof}
\begin{remark}
The reason why Remark~\ref{RM:2} holds is because if $t$ is a monomial of $F_{2m-k}$, then $deg(t) + 2m-k$ is the number of $1$-cycles in an $\eta \in \mathcal{U}_{G'}$ such that $[t \lambda^{2m-k}] w'(\eta) \neq 0$. Thus the remaining cycles must be $2$-cycles, and so $k-deg(t)$ must be even; that is, parities of $k$ and $deg(t)$ must agree. 
\end{remark}

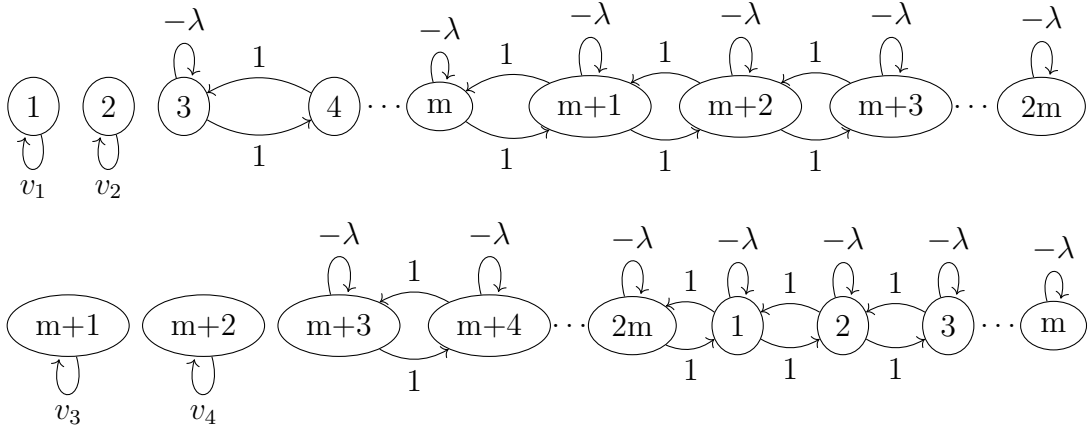
\begin{figure}[ht]
 \begin{tikzpicture}[->]
    \node[ellipse,draw] (a) at (0,0) {1};
    \node[ellipse,draw] (b) at (1,0) {2};
    \node[ellipse,draw] (c) at (2,0) {3};
    \node[ellipse,draw] (l) at (4,0) {4};
    \node[ellipse,draw] (d) at (5.4,0) {m};
    \node[ellipse,draw] (e) at (7.4,0) {m+1};   
    \node[ellipse,draw] (f) at (9.4,0) {m+2};   
    \node[ellipse,draw] (g) at (11.4,0) {m+3};   
    \node[ellipse,draw] (h) at (13.4,0) {2m};   
    
    \node at (4.7,0) {\dots};
    \path (c) edge[bend right]              node[below] {1} (l);
    \path (l) edge[bend right]           node[above] {1} (c);
    \path (d) edge[bend right]              node[below] {1} (e);
    \path (e) edge[bend right]           node[above] {1} (d);
    \path (e) edge[bend right]              node[below] {1} (f);
    \path (f) edge[bend right]           node[above] {1} (e);
    \path (f) edge[bend right]              node[below] {1} (g);
    \path (g) edge[bend right]           node[above] {1} (f);
    \node at (12.5,0) {\dots};
    \path (a) edge [loop below] node {$v_1$} (a);
    \path (b) edge [loop below] node {$v_2$} (b);
    \path (c) edge [loop above] node {$-\lambda$} (c);
    \path (d) edge [loop above] node {$-\lambda$} (d);
    \path (e) edge [loop above] node {$-\lambda$} (e);
    \path (f) edge [loop above] node {$-\lambda$} (f);
    \path (g) edge [loop above] node {$-\lambda$} (g);
    \path (h) edge [loop above] node {$-\lambda$} (h);
  \end{tikzpicture} 

 \begin{tikzpicture}[->]
    \node[ellipse,draw] (a) at (0,0) {m+1};
    \node[ellipse,draw] (b) at (1.8,0) {m+2};
    \node[ellipse,draw] (c) at (3.6,0) {m+3};
    \node[ellipse,draw] (l) at (5.6,0) {m+4};
    \node[ellipse,draw] (d) at (7.5,0) {2m};
    \node[ellipse,draw] (e) at (8.9,0) {1};   
    \node[ellipse,draw] (f) at (10.3,0) {2};   
    \node[ellipse,draw] (g) at (11.7,0) {3};   
    \node[ellipse,draw] (h) at (13.1,0) {m};   
    
    \node at (6.7,0) {\dots};
    \path (c) edge[bend right]              node[below] {1} (l);
    \path (l) edge[bend right]           node[above] {1} (c);
    \path (d) edge[bend right]              node[below] {1} (e);
    \path (e) edge[bend right]           node[above] {1} (d);
    \path (e) edge[bend right]              node[below] {1} (f);
    \path (f) edge[bend right]           node[above] {1} (e);
    \path (f) edge[bend right]              node[below] {1} (g);
    \path (g) edge[bend right]           node[above] {1} (f);
    \node at (12.4,0) {\dots};
    \path (a) edge [loop below] node {$v_3$} (a);
    \path (b) edge [loop below] node {$v_4$} (b);
    \path (c) edge [loop above] node {$-\lambda$} (c);
    \path (d) edge [loop above] node {$-\lambda$} (d);
    \path (e) edge [loop above] node {$-\lambda$} (e);
    \path (f) edge [loop above] node {$-\lambda$} (f);
    \path (g) edge [loop above] node {$-\lambda$} (g);
    \path (h) edge [loop above] node {$-\lambda$} (h);
    \path (l) edge [loop above] node {$-\lambda$} (l);
  \end{tikzpicture} 

\caption{(Above) $G'$ after fixing the cycles $v_1$ and $v_2$ and removing cycles $v_3$ and $v_4$. (Below) $G'$ after fixing the cycles $v_3$ and $v_4$ and removing cycles $v_1$ and $v_2$. Notice how these are the same digraphs, just relabeled; thus indeed the coefficients of $v_1v_2$ and $v_3v_4$ are the same.}
\label{Fig:3}
\end{figure}




Now we can prove Theorem~\ref{THM:3.2}.

\subsection{Proof of Theorem~\ref{THM:3.2}}

\begin{lemma} Let \[v = (v_1, v_2, 0, \dots, 0, v_{3}, v_{4}, 0, \dots, 0)\]
with  $v_1 = 1+i$, $v_2 = 1-i$, $v_3 = -1+i$, and $v_4 = -1-i$. Then $F_{2m-k}(v)=0$ for all $k\in[2m]$.
\end{lemma}
We split the proof into two cases.
\begin{enumerate}[label=(\roman*)]
    \item $k$ \emph{is odd.}

Notice that $v_1 = -v_4$ and $v_2 = -v_3$, thus we have that
\begin{equation*}
\left\{
\begin{array}{l}
 v_1 + v_2 + v_3 + v_4 = 0. \\
 v_1v_2v_3+v_2v_3v_4+v_1v_3v_4+v_1v_2v_4 = 0.\\
\end{array}\right.
\end{equation*}


Combined with Theorem \ref{thm:Fnformula}, this completes the proof in the case that $k$ is odd.
\item $k$ \emph{is even} ($k = 2 \ell$).

With the given choice of $v$ we have
\begin{equation*}
\left\{
\begin{array}{l}
v_1v_2v_3v_4 = 4\\
v_1 v_2 + v_3 v_4 =4\\
v_1 v_3 + v_2 v_4=-4\\
v_1 v_4 + v_2 v_3 = 0.\\
\end{array}\right.
\end{equation*}

Thus, according to Theorem \ref{thm:Fnformula} we are left to show that 
\begin{equation}\label{eq:combinatorialident}
S(2m-2,\ell-1)= \sum_{i = 0}^{\ell-2} S(m-2, i) S(m-2, \ell-2-i) + \sum_{i = 0}^{\ell-1} S(m-1, i) S(m-1 , \ell-1-i).
\end{equation}

To show that \eqref{eq:combinatorialident} holds, we will prove that both sides count the same collection of objects. Indeed, according to the definition of $S(2m-2,\ell-1)$, the left-hand side counts the number of disjoint cycle covers of the digraph $J_{(2m-2)}$ that have $\ell-1$ $2$-cycles. We now claim that the right-hand side enumerates the same collection. Indeed, suppose we have a $J_{(2m-2)}$ digraph and we wish to count the number of disjoint cycle covers with $\ell-1$ $2$-cycles. We may partition our collection of disjoint cycle covers into two sets: those that contain the cycle $((m-1,m),(m,m-1))$, denoted by $\mathcal{C}_1$, and those that do not, which we call $\mathcal{C}_2$. 

To count $\mathcal{C}_1$ we can start by removing the vertices $m-1$ and $m$ from $J_{(2m-2)}$, we are left with the disjoint union of two $J_{(m-2)}$ digraphs, which we denote $G_1'$ and $G_2'$ respectively. In particular, the number of elements in $\mathcal{C}_1$ is the number of disjoint cycle covers of the disjoint union of $G_1'$ and $G_2'$ with $\ell-2$ $2$-cycles. It readily follows that   $$|\mathcal{C}_1|=\sum_{i = 0}^{\ell-2} S(m-2, i) S(m-2, \ell-2-i).$$
We now proceed to count the elements of $\mathcal{C}_2$. Note that the edges $(m-1,m)$ and $(m,m-1)$ cannot appear in any element of $\mathcal{C}_2$. Removing these edges from $J_{(2m-2)}$, it follows that $\eta \in \mathcal{C}_2$ if and only if $\eta$ is a disjoint cycle cover of $G'_1\dot{\bigsqcup} G'_2$ with $\ell-1$ $2$-cycles where $G'_1$ and $G'_2$ are two $J_{(m-1)}$ digraphs and $\dot{\bigsqcup}$  denotes a disjoint union. It readily follows that $$|\mathcal{C}_2|=\sum_{i = 0}^{\ell-1} S(m-1, i) S(m-1 , \ell-1-i).$$ 
As $S(2m-2,\ell-1)=|\mathcal{C}_1|+|\mathcal{C}_2|$ we have established \eqref{eq:combinatorialident}, finishing the proof.
\end{enumerate}
\hfill \qed

\section*{Acknowledgments}
This research was conducted as part of the ongoing Undergraduate Research Program,``STODO" (Spectral Theory Of Differential Operators), at Texas A\&M University. We are grateful for the support provided by the College of Arts and Sciences Undergraduate Research Program at Texas A\&M University. 
This work was also partially supported by NSF
  DMS-2000345, DMS-2052572, and DMS-2246031. R. Matos is partially supported by CNPq's grant ``Projeto Universal'' (402952/2023-5)
  
	
\bibliographystyle{abbrv} 
	\bibliography{REU}
	
\end{document}